\documentclass[12pt]{amsart}
\usepackage{euscript,epsf,amssymb,xr}

\let\oldmarginpar\marginpar
\renewcommand\marginpar[1]
{\oldmarginpar{\tiny\bf \begin{flushleft} #1 \end{flushleft}}}

\input xy
\xyoption{all}

\newcommand{\la}{\langle}
\newcommand{\ra}{\rangle}

\newtheorem{theorem}{\bf Theorem}[section]
\newtheorem{lemma}[theorem]{\bf Lemma}

\newcommand{\PP}{{\Bbb P}}

\newcommand{\RR}{{\Bbb R}}

\newcommand{\ZZ}{{\Bbb Z}}

\newcommand{\klie}{{\frak k}}

\newcommand{\slie}{{\frak s}}

\newcommand{\ggreat}{>\kern-.7ex>}
\newcommand{\ssmall}{<\kern-.7ex<}
\newcommand{\qu}{/\kern-.7ex/}
\newcommand{\exh}{\to\kern-1.8ex\to}

\newcommand{\gG}{{\EuScript{G}}}

\newcommand{\lL}{{\EuScript{L}}}

\newcommand{\nN}{{\EuScript{N}}}

\newcommand{\sS}{{\EuScript{S}}}

\newcommand{\uU}{{\EuScript{U}}}

\newcommand{\GL}{\operatorname{GL}}

\newcommand{\Aut}{\operatorname{Aut}}

\newcommand{\Diff}{\operatorname{Diff}}
\newcommand{\Gr}{\operatorname{Gr}}

\newcommand{\id}{\operatorname{id}}

\newcommand{\Isom}{\operatorname{Isom}}

\renewcommand{\O}{\operatorname{O}}
\renewcommand{\vert}{\operatorname{vert}}

\newcommand{\ov}{\overline}

\newcommand{\imag}{{\mathbf{i}}}

\title[Isometry groups of closed Lorentz 4-manifolds are Jordan]
{Isometry groups of closed Lorentz 4-manifolds are Jordan}

\author{Ignasi Mundet i Riera}
\address{Facultat de Matem\`atiques i Inform\`atica\\
Universitat de Barcelona\\
Gran Via de les Corts Catalanes 585\\
08007 Barcelona \\
Spain}
\email{ignasi.mundet@ub.edu}

\date{January 11, 2019}


\thanks{This work has been partially supported by the (Spanish) MEC Project
MTM2015-65361-P}

\begin{document}

\maketitle

\begin{abstract}
We prove that for any closed Lorentz $4$-manifold $(M,g)$ the isometry group
$\Isom(M,g)$ is Jordan. Namely, there exists a constant $C$ (depending on $M$ and $g$)
such that any finite subgroup $\Gamma\leq\Isom(M,g)$ has an abelian subgroup
$A\leq\Gamma$ satisfying $[\Gamma:A]\leq C$.
\end{abstract}

\section{Introduction}

It is well known that the isometry group of any closed\footnote{Recall that a manifold is closed if it is compact and has no boundary.}  Lorentz manifold
is a Lie group\footnote{See e.g. Lemma \ref{lemma:convergencia} below for some references.}. This is in
fact true for metrics of arbitrary signature. The case of Riemannian metrics, proved originally
by Myers and Steenrod \cite{MyS}, is probably the most well known example. But there is an important aspect in which the Riemannian (or definite) case differs from the other ones: while the isometry group of any closed Riemannian manifold is compact,
if $g$ is an indefinite metric on a closed manifold $M$
then its isometry group $\Isom(M,g)$ may fail to be compact and may even have infinitely many connected
components (see e.g. \cite[\S 1]{D} for some examples which are Lorentz metrics).

Many papers have been written in the past decades studying closed Lorentz manifolds
whose isometry group is noncompact. In particular, it is by now well understood
which noncompact connected Lie groups may appear as the
identity component of the isometry group of some closed Lorentz manifold,
see \cite{AS1,AS2,G,K,Ze,Zi}.
However, the possible structure of the entire group of isometries remains much more mysterious, although some partial results are now available \cite{PZ}.

Our purpose in this note is to study a particular question related to the finite groups of isometries of closed Lorentz manifolds, which is particularly interesting when the entire isometry group contains infinitely many components.

Let us recall that a group $\gG$ is said to be $C$-Jordan (where $C$ is a positive
number) if every finite subgroup $G\leq\gG$ has an abelian subgroup $A\leq G$ satisfying $[G:A]\leq C$.
We say that $\gG$ is Jordan if it is $C$-Jordan for some $C$.
Roughly speaking, a group is Jordan if there is a bound on {\it how much nonabelian} its finite subgroups can be.

A classical theorem of Camille Jordan states that $\GL(n,\RR)$ is Jordan for every $n$. One can easily deduce from
this fact that any (finite dimensional) Lie group with finitely many connected components is Jordan,
see e.g. \cite{BW,Po3}.
So all compact Lie groups are Jordan, and in particular the
isometry group of any closed Riemannian manifold is Jordan.
The analogous question of whether isometry groups of closed Lorentz manifolds are Jordan
is much more interesting, because they can have infinitely many
connected components.

Of course, if a closed manifold $M$ has Jordan diffeomorphism group then a fortiori for every Lorentz
metric $g$ on $M$ the isometry group $\Isom(M,g)$ is going to be Jordan. It is known that the diffeomorphism
group of any closed manifold of dimension at most three is Jordan \cite{M0,Z}, but in any
dimension bigger than three
there exist closed manifolds whose diffeomorphism group is not Jordan \cite{CPS,M6} (many more things are
known in dimension four, and also in higher dimensions, see \cite{MS}).
Hence, four is the lowest dimension
in which the problem of Jordanness of isometry groups of closed
Lorentz manifolds is nontrivial. This is our main theorem:

\begin{theorem}
\label{thm:main}
The isometry group of any closed Lorentz 4-manifold is Jordan.
\end{theorem}

Actually, the only known examples of closed 4-manifolds with non Jordan
group of diffeomorphisms are $T^2\times S^2$ and the nontrivial $S^2$-fibration over $T^2$.
Both manifolds admit Lorentz metrics whose isometry group has infinitely many components,
as follows easily from the examples in \cite[\S 1]{D}. Moreover, we have:

\begin{theorem}
\label{thm:main-2}
Let $M$ be either $T^2\times S^2$ and the nontrivial $S^2$-fibration over $T^2$.
For any positive number $C$ there exists a Lorentz metric $g$ on $M$ and a finite
subgroup $\Gamma\leq\Isom(M,g)$ all of whose abelian subgroups $A\leq G$ satisfy
$[G:A]>C$.
\end{theorem}
\begin{proof}
Let us briefly recall the main construction in \cite{CPS}.
Choose an orientation of $T^2$.
For any natural number $n$ denote by $\Gamma_n$ the Heisenberg group
$\Gamma_n=\la a,b,c\mid a^n=b^n=c^n=[a,c]=[b,c]=1,\,[a,b]=c\ra$,
and let $L_n\to T^2$ be a complex line bundle of degree $n$. There is an
effective smooth action of $\Gamma_n$ on $L_n$ by line bundle automorphisms
(see \cite{Mum,M5,M6}) which gives rise to an action on the projectivisation
$\PP(L_n\oplus L_n^{-1})$; the latter is effective provided $n$ is odd. The action
of $\Gamma_n$ on $L_n$ lifts an action on $T^2$ which can be described in terms
of an identification $T^2\simeq S^1\times S^1$ by the formulas
$$a\cdot(\theta_1,\theta_2)=(e^{2\pi\imag/n}\theta_1,\theta_2), \quad
b\cdot(\theta_1,\theta_2)=(\theta_1,e^{2\pi\imag/n}\theta_2), \quad
c\cdot(\theta_1,\theta_2)=(\theta_1,\theta_2).$$
For any $n$ there is a diffeomorphism $\PP(L_n\oplus L_n^{-1})\simeq T^2\times S^2$.
Since any abelian subgroup $A\leq\Gamma_n$ satisfies $[\Gamma_n:A]\geq n$ (see \cite{Zar}),
this proves that $\Diff(T^2\times S^2)$ is not Jordan.

To prove the theorem let us endow $\PP(L_n\oplus L_n^{-1})$ with a $\Gamma_n$-invariant
Lorentz metric. Consider the metric $g_T=d\theta_1^2-d\theta_2^2$ on $T^2$. The previous
formulas imply that $g_T$ is invariant under the action of $\Gamma_n$ on $T^2$. Let us denote
$X=\PP(L_n\oplus L_n^{-1})$ and let $\pi:X\to T^2$ be the projection. There is an exact
sequence of vector bundles on $X$
$$0\to T^{\vert}X\to TX\to \pi^*TT^2\to 0,$$
where $T^{\vert}X$ is the vertical tangent bundle on $X\to T^2$.
Choose a $\Gamma_n$-invariant connection on $L_n$. This induces an invariant connection on $TX$,
hence an equivariant isomorphism $TX\simeq T^{\vert}X\oplus \pi^*TT^2$. Let $g_V$ be
a $\Gamma_n$-invariant euclidean metric on $T^{\vert}X$. Then $g_V+\pi^*g_T$ is a
$\Gamma_n$-invariant Lorentz metric on $X$.
\end{proof}

Of course we have a completely analogous result for the nontrivial $S^2$ fibration on $T^2$.

The proof of Theorem \ref{thm:main} is given in Section \ref{s:proof-main-thm}.
Before that, we prove in Section \ref{s:group-actions} a technical result on smooth
actions of compact groups.

\section{Compact group actions and fixed submanifolds}
\label{s:group-actions}

The following result is probably well known in the context of compact transformation
groups. Since we have not found it explicitly written in the literature we give a detailed proof.

\begin{lemma}
\label{lemma:fixed-submanifolds}
Let $K$ be a compact Lie group acting smoothly on a closed manifold
$M$. Let $\sS$ be the set of all submanifolds $X\subseteq M$ for which
there is some element $k\in K$ such that $X$ is a connected component
of the fixed point set $M^k$. The action of $K$ on $M$ induces an action
on the set $\sS$. Then the number of $K$-orbits in $\sS$ is finite.
\end{lemma}
\begin{proof}
We use ascending induction on the dimension of $M$. The case $\dim M=0$
being obvious, we assume that $\dim M>0$ and that the lemma holds true
for manifolds of dimension less than $\dim M$. We will argue by contradiction,
so let us assume that there is an infinite sequence of submanifolds of $M$, $(X_i)$,
and elements of $K$, $(k_i)$, with $X_i$ a connected component of $M^{k_i}$,
in such a way that for any $i\neq j$ there is no $k\in K$ such that $X_j=kX_i$.

Choose a $K$-invariant Riemannian metric $\rho$ on $M$. Since $M$ is compact, we may
assume, replacing $(X_i)$ and $(k_i)$ by subsequences if necessary, that there
exist points $x_i\in X_i$ such that $(x_i)$ converges to some $x\in M$. Let $\klie(x)\subseteq T_xM$ be
the tangent space to the $K$-orbit through $x$, and let $\slie$ be the $\rho$-orthogonal
of $\klie(x)$. The isotropy group $K_x$ at $x$ acts on $T_xM$ preserving $\slie$,
and the map $\ov{e}:K\times\slie\ni (k,s)\mapsto k\exp_x^{\rho}s$ satisfies $\ov{e}(kk',s)=\ov{e}(k,k's)$,
so it descends to a map $e:K\times_{K_x}\slie\to M$. The projection map $\pi:K\times_{K_x}\slie\to K/K_x$ gives a structure of $K$-equivariant vector bundle on $K\times_{K_x}\slie$, the group $K$ acting naturally on the left on the total space and the base.
As is well known, the restriction of $e$ to an invariant neighborhood $\nN$
of the zero section of this bundle is
an equivariant embedding (this is the slice theorem for smooth compact group actions).
We identify $\nN$ with its image in $M$.

Assume, for the remainder of the proof, that $i$ is big enough so that $x_i$ is contained in $\nN$.
Suppose that $\pi(x_i)=h_iK_x$, with $h_i$. Replacing $X_i$ by $h_i^{-1}X_i$ (and $k_i$
by $h_i^{-1}k_ih_i$) we may assume
that $\pi(x_i)=K_x$. Now, the stabilizer of any $y\in\pi^{-1}(K_x)$ is contained in $K_x$,
so in particular $k_i\in K_x$. This implies that $x\in X_i$. Let $S$ be the unit sphere in
$T_xM$. Applying the induction hypothesis to the action of $K_x$ on $S$ and looking at
the sequence of submanifolds $(T_xX_i\cap S)$ of $S$ we may conclude that there exists some $i\neq j$
and some $k\in K_x$ such that $T_xX_j\cap S=k(T_xX_i\cap S)$, which implies that
$T_xX_j=kT_xX_i$ and hence $X_j=kX_i$ (this is standard: since the exponential map is $K_x$-equivariant,
we have $\exp^{\rho}_x(T_xX_j)\subseteq X_j$, and similarly $\exp^{\rho}_x(kT_xX_i)\subseteq kX_i$; hence $kX_i\cap X_j$ has nonempty interior both in $kX_i$ and $X_j$; the same argument allows
to prove more generally that $kX_i\cap X_j$ is open in $kX_i$ and $X_j$, and since the intersection
is also closed and both $kX_i$ and $X_j$ are connected, it follows that $X_j=kX_i$).
We have reached a contradiction, so this concludes the proof of the lemma.
\end{proof}

\section{Normal bundles of surfaces fixed by periodic isometries}
\label{s:proof-main-thm}

To begin with, let us observe that to prove Theorem \ref{thm:main}
it suffices to consider closed oriented Lorentz
$4$-manifolds (see \cite[\S 2.3]{M0} or Lemma 2.1 in \cite{MS}).

Let $(M,g)$ be a closed oriented Lorentz 4-manifold. Let $F_gM$ be the principal $\O(3,1)$-bundle
of $g$-orthonormal frames of $M$. Any isometry $\phi$ of $(M,g)$ induces a diffeomorphism of $F_gM$
which we denote by $D\phi$, and we denote by $d\phi$ the diffeomorphism of $TM$ induced by $\phi$.

\begin{lemma}
\label{lemma:convergencia}
Suppose that $(\phi_i)$ is a sequence of isometries of $(M,g)$, and that there exists
a converging sequence $(z_i)\subset F_gM$ such that $(D\phi_i(z_i))$ converges somewhere in $F_gM$.
Then $(\phi_i)$ has a subsequence converging in $\Isom(M,g)$.
\end{lemma}
\begin{proof}
Let $\nabla$ be the Levi--Civita connection of $g$, and denote by $\Aut(M,\nabla)$ the
group of diffeomorphisms of $M$ that preserve $\nabla$. The group $\Aut(M,\nabla)$ has a structure of
finite dimensional Lie group \cite{CK,HM}, and with respect to this structure $\Isom(M,g)$ is a closed subgroup of $\Aut(M,\nabla)$ (hence $\Isom(M,g)$ is a Lie group). Consequently, it
suffices to prove the statement for sequences in $\Aut(M,\nabla)$.

Suppose that $z_i\to z$ and  that $D\phi_i(z_i)\to w$.
Denote by $\pi:F_gM\to M$ the projection. Let
$x_i=\pi(z_i)$ and $x=\pi(z)$, so that $x_i\to x$ and
\begin{equation}
\label{eq:phi-x-convergeix}
\phi_i(x_i)\to \pi(w).
\end{equation}
For any $y\in M$ and any tangent vector
$s\in T_yM$ sufficiently close to $0$ denote by $\exp^{\nabla}_y(v)\in M$ the image of the exponential map. For any big enough $i$ there is some $v_i\in T_xM$ such that $x_i=\exp^{\nabla}_x(v_i)$.
Furthermore, $v_i\to 0$. Let $u_i\in T_{x_i}M$ be the parallel transport of $v_i$ along the
curve $\gamma_i:[0,1]\ni t\mapsto \exp^{\nabla}_x(tv_i)$. Then $x=\exp^{\nabla}_{x_i}(-u_i)$ and
$u_i\to 0\in T_xM$.
Since $D\phi_i(z_i)$ converges in $F_gM$ and $u_i\to 0\in T_xM$, we have
\begin{equation}
\label{eq:vector-tangent-convergeix}
d\phi_i(u_i)\to 0\in T_{\pi(w)}M.
\end{equation}
Let $\zeta_i\in \pi^{-1}(x_i)\in F_gM$ be the parallel transport of $z$ along $\gamma_i$.
We have $\zeta_i=z_iX_i$ for some
$X_i\in\O(3,1)$, and
\begin{equation}
\label{eq:X-convergeix}
X_i\to 1
\end{equation}
because $z_i\to z$. Let $w_i$ be the parallel transport of $D\phi_i(z_i)$ along the
curve
$$[0,1]\ni t\mapsto \phi_i(\exp_{x_i}^{\nabla}(-tu_i))=\exp_{\phi_i(x_i)}^{\nabla}(-td\phi_i(u_i)).$$
Since $\phi_i$ preserves $\nabla$, $D\phi_i(z)$ is equal $w_iX_i$. Combining (\ref{eq:phi-x-convergeix}),
(\ref{eq:vector-tangent-convergeix})
and (\ref{eq:X-convergeix}) we conclude that $D\phi_i(z)\to w$.
At this point the lemma follows from \cite[Lemma 5]{HM}.
\end{proof}

The space of reductions of the structure group of $F_gM$ to the maximal
compact subgroup $K=\O(3,\RR)\times\ZZ/2<\O(3,1)$ can be identified with the sections
of the bundle $F_gM/K$, whose fibers can be identified with $\O(3,1)/K$, which is
contractible. It follows that reductions to $K$ exist. Take one
reduction, and fix it for the entire argument. The choice of a reduction
amounts to giving an isomorphism
$TM\simeq L\oplus S$, where $L,S\to M$ are two Euclidean vector bundles of ranks
$1$ and $3$ respectively, with the property that for any $u=(l,s)\in L\oplus S$ we have
$$g(u,u)=\|s\|^2-\|l\|^2,$$
where $\|\cdot\|$ denotes the Euclidean norms on $L$ and $S$.
Let $g_R$ denote the Riemannian metric on $M$ given by the Euclidean structures on $L$
and $S$.

For any $p\in M$ we denote by $\Isom(T_pM,g)$ (resp. $\Isom(T_pM,g_R)$) the linear isomorphisms
which preserve $g(p)$ (resp. $g_R(p)$).

\begin{lemma}
\label{lemma:N-dins-V}
Let $\phi\in\Isom(T_pM,g)$ be a finite order isometry which is not an involution.
Suppose that there is a $\phi$-invariant splitting
$T_pM=F\oplus N$ in $g$-orthogonal planes such that $\phi|_F=\id_F$ and $N\subset V$. Then
$\phi\in\Isom(T_pM,g_R)$.
\end{lemma}
\begin{proof}
It suffices to consider the case $\phi\neq\id$.
Let $f\in F$ be any element. There exists some nonzero $n\in N$ such that
$g(f,n)=0$, since the $g$-orthogonal of $f$ has dimension $\geq 3$. Since $\phi$ has finite order
and is not an involution, and $\dim N=2$, $n$ and $\phi(n)$ are linearly independent. We have
$0=g(f,n)=g(\phi(f),\phi(n))=g(f,\phi(n))$, and hence $N$ is contained in the $g$-orthogonal
of $f$. Letting $f$ run along all elements of $F$ we conclude that $F,N$ are $g$-orthogonal.
In particular, $F\cap V$ is equal to the $g_R$-orthogonal of $N$ in $V$. Hence, we may take a
$g_R$-orthogonal basis $e_1,e_2,e_3,e_4$ of $T_pM$ with respect to which $N=\la e_1,e_2\ra$
and $F=\la e_3,e_4\ra$. This proves the lemma.
\end{proof}

By Theorem 1.4 in \cite{MS}, Theorem \ref{thm:main} follows from:

\begin{theorem}
\label{thm:bound-degree-normal-bundle}
There exists a constant $C$ with the following property. Suppose that the order of
$\phi\in\Isom(M,g)$
is finite and bigger than $2$,
and that the fixed point set $M^{\phi}$ has a connected component
$\Sigma\subset M$ which is an embedded orientable surface.
Denote by $\nu\to\Sigma$ be the normal bundle, and choose
orientations of $\Sigma$ and $\nu$. Then $|\deg\nu|\leq C$.
\end{theorem}

Note that  the normal bundle $\nu$ is orientable because by assumption both $\Sigma$ and $M$ are
orientable. Before proving Theorem \ref{thm:bound-degree-normal-bundle} we will prove the following
technical result.

\begin{lemma}
\label{lemma:compact}
Let $S$ be the set of finite order elements $\phi\in\Isom(M,g)$ of order bigger than $2$
such that $M^{\phi}$ contains a connected component which is an orientable embedded surface
whose normal bundle has nonzero degree. Then $S$ is relatively compact.
\end{lemma}

To define the degree of the normal bundle of an orientable and connected
embedded surface in $M$ one has to choose orientations of the surface
and the bundle, but the condition that the degree is nonzero is independent
of the choices.

\begin{proof}
Let $(\phi_i)$ be a sequence of elements in $S$. We are going to prove that $(\phi_i)$
has a converging subsequence. Suppose that $\Sigma_i$ is a connected component of
$M^{\phi_i}$ and that the normal bundle $\nu_i\to\Sigma_i$ has nonzero degree.
We identify $\nu_i$ with a subbundle of $TM|_{\Sigma_i}$ in the usual way. To be precise,
for any $x\in\Sigma_i$ the derivative of $\phi_i$ gives a linear automorphism
$d\phi_i\in \Aut(T_xM)$ of finite order, and there is a splitting in $d\phi_i$-invariant subspaces
$T_xM=A_x\oplus B_x$, where $d\phi_i$ acts trivially on $A_x$ and the restriction of $d\phi_i$
to $B_x$ does not have the eigenvalue $1$.
We have $A_x=T_x\Sigma$, and as $x$ moves along $\Sigma_i$ the spaces $B_x$ span
a real vector bundle of rank $2$ which can be identified with $\nu_i$.

We claim that there is some point $x_i\in\Sigma_i$ such that the fiber of $\nu_i$
over $x_i$ is contained inside $S_{x_i}$. Indeed, if this were not the case then
$\Lambda_i:=\nu_i\cap S|_{\Sigma_i}$ would be a (real) line subbundle of $\nu_i$.
This would force $\nu_i$ to have degree $0$, a contradiction.

It follows from Lemma \ref{lemma:N-dins-V} that $(d\phi_i)_{x_i}\in\Isom(T_{x_i}M,g_R)$.
In other words, if we denote by $F_{g_R}M\subset F_gM$ the space of $g_R$-orthogonal
frames, for any $i$ the intersection $F_{g_R}M\cap D\phi_i(F_{g_R}M)\neq\emptyset$.
Since $F_{g_R}M$ is compact, we may assume, passing to a subsequence if necessary, that
there exist points $z_i\in F_{g_R}M$ such that $D\phi_i(z_i)$ converges in $F_{g_R}M$.
Applying Lemma \ref{lemma:convergencia} we conclude that $\phi_i$ has a converging
subsequence.
\end{proof}

We are now ready to prove Theorem \ref{thm:bound-degree-normal-bundle}.
Let $\Isom^S(M,g)$ be the union of all connected components of $\Isom(M,g)$ that
contain points of $S$. Since $S$ is relatively compact, $\Isom^S(M,g)$ has
finitely many connected components. Note that there is no reason to assume that
$\Isom^S(M,g)$ is a subgroup of $\Isom(M,g)$. The collection of
all subgroups of $\Isom^S(M,g)$ whose identity component coincides with that
of $\Isom(M,g)$ is obviously finite.
Denote these subgroups by $G_1,\dots,G_r$.

Each of the groups $G_i$ has finitely many connected components, so we
can choose, for every $i$, a maximal compact subgroup $K_i$ of
$G_i$, which has the property that any compact subgroup of $G_i$ is conjugate
to a subgroup of $K_i$. The existence of $K_i$ is pretty standard if $G_i$
is connected; for the general case, see e.g. \cite[Theorem 14.1.3]{HN}.

Let $\sS_i$ be the collection of all embedded surfaces $\Sigma$ for which there
is some $k\in K_i$ such that $\Sigma$ is a connected component of $X^k$.
For every $\Sigma\in\sS_i$ let $\nu_{\Sigma}$ denote the normal bundle of $\Sigma$.
By Lemma \ref{lemma:fixed-submanifolds} the set of $K_i$-orbits in $\sS_i$ is
finite. If $\Sigma,\Sigma'\in\sS_i$ satisfy $\Sigma'=k\Sigma$ for some $k\in K_i$
then the bundles $\nu_{\Sigma}\to\Sigma$ and $\nu_{\Sigma'}\to\Sigma'$ are isomorphic,
so $|\deg\nu_{\Sigma}|=|\deg\nu_{\Sigma'}|$. It follows that
the number $C_i=\max\{|\deg\nu_{\Sigma}|:\Sigma\in\sS_i\}$ is finite.

Set
$C=\max\{C_1,\dots,C_r\}$. We claim that this number $C$ has the property stated
in Theorem \ref{thm:bound-degree-normal-bundle}. Indeed, if $\phi\in\Isom(M,g)$
has finite order bigger than $2$ and $\Sigma\subset M$ is an embedded surface which is a connected
component of $M^{\phi}$ with $|\deg\nu_{\Sigma}|\neq 0$ then $\phi$ and all its powers
belong to $S$. It follows that $\phi\in G_i$ for some $i$. Since $\phi$ has finite order, it
is contained in a compact subgroup of $G_i$, so a conjugate of $\phi$ belongs to $K_i$.
This finishes the proof of the theorem.

\end{document}